\documentclass[11pt,twoside]{article}
\usepackage[margin=1.3in]{geometry} 

\usepackage{amsfonts,amsmath,amssymb,amsthm}
\usepackage{CJK,CJKnumb,CJKulem,cases,cite,color,curves}
\usepackage{bm,dsfont,extarrows,enumerate,fontenc,graphicx,ifthen,latexsym,listings,makeidx,mathrsfs,psfrag,times,xcolor}
\usepackage[colorlinks,citecolor=green,linkcolor=red]{hyperref}
\usepackage[title]{appendix}

\let\oldbibliography\thebibliography
\renewcommand{\thebibliography}[1]{\oldbibliography{#1}\setlength{\itemsep}{0pt}}

\makeindex

\newtheorem{theorem}{Theorem}[section]
\newtheorem{lemma}[theorem]{Lemma} 
\newtheorem{proposition}[theorem]{Proposition}

\begin{document}

\title{\textbf{The direct moving sphere for fractional Laplace equation}}

\author {Congming Li\thanks{Email: congming.li@sjtu.edu.cn. School of Mathematical Sciences, Shanghai Jiao Tong University, Shanghai 200240, China. Partially supported by NSFC-12031012, NSFC-12250710674 and the Institute of Modern Analysis-A Frontier Research Center of Shanghai.}  
\hspace{.2in} Meiqing Xu\thanks{Email: xmq157@sjtu.edu.cn. School of Mathematical Sciences, 
Shanghai Jiao Tong University, Shanghai 200240, China. Partially supported by NSFC-12031012, NSFC-12250710674 and the Institute of Modern Analysis-A Frontier Research Center of Shanghai.}
\hspace{.2in}Hui Yang\thanks{Email: hui-yang@sjtu.edu.cn. School of Mathematical Sciences, 
Shanghai Jiao Tong University, Shanghai 200240, China. Partially supported by NSFC-12301140 and the Institute of Modern Analysis-A Frontier Research Center of Shanghai.} 
\hspace{.2in} Ran Zhuo\thanks{E-mail: zhuoran1986@126.com. Department of Mathematics, Shanghai Normal University, Shanghai 200234, China. Partially supported by the Institute of Modern Analysis-A Frontier Research Center of Shanghai.} }

\maketitle
\begin{abstract}
\noindent  
This paper works on the direct method of moving spheres and establishes a Liouville-type theorem for the fractional elliptic equation
\[
 (-\Delta)^{\alpha/2} u =f(u)  ~~~~~~  \text{in } \mathbb{R}^{n}
\]
with general non-linearity. One of the key improvement over the previous work is that we do not require the usual Lipschitz condition. In fact, we only assume the structural condition that $f(t) t^{- \frac{n+\alpha}{n-\alpha}}$ is monotonically decreasing. This differs from the usual approach such as Chen-Li-Li (Adv. Math. 2017), which needed the Lipschitz condition on $f$, or Chen-Li-Zhang (J. Funct. Anal. 2017), which relied on both the structural condition and the monotonicity of $f$. We also use the direct moving spheres method to give an alternative proof for the Liouville-type theorem of the fractional Lane-Emden equation in a half space. Similarly, our proof does not depend on the integral representation of solutions compared to existing ones. The methods developed here should also apply to problems involving more general non-local operators, especially if no equivalent integral equations exist. 

\medskip 

 
\medskip

\end{abstract}

\section{Introduction} 
The fractional Laplacian has been extensively used in the study of physics, biology, mechanics and economics.
The operator is the infinitesimal generator of a stable L\'{e}vy process--a jump process (\cite{B}). In  L\'{e}vy flight diffusion process, the fractional Laplacian can be used to characterize memory effects and long-distance diffusion processes (\cite{Ca,MJ}). The operator is used in the biological population dynamics model--Fisher-KPP equations to accurate localization of solutions (\cite{CR,FY,SV}) and to model the dynamics in the Hamiltonian chaos in astrophysics(\cite{BG,Za}). It also has various applications in probability and finance (\cite{A,Da,Ke}).

The fractional Laplacian is a non-local integro-differential operator taking the following form:
\begin{equation}
(-\Delta)^{\alpha/2} u(x) = C_{n,\alpha} \, P.V. \int_{\mathbb{R}^n}
\frac{u(x)-u(z)}{|x-z|^{n+\alpha}} dz,\label{Ad7}
\end{equation}
where $0< \alpha<2$ and P.V. stands for the Cauchy principal value. Define
$$\mathcal{L}_{\alpha}=\left\{u: \mathbb{R}^n\rightarrow \mathbb{R}  \Big|  \int_{\mathbb{R}^n}\frac{|u(x)|}{1+|x|^{n+\alpha}} \, d x <\infty \right\}.$$
Then the operator $(-\Delta)^{\alpha/2} u $ is well-defined for the functions $u\in \mathcal{L}_{\alpha} \cap C^{1,1}$.

The method of moving planes is one of the most important methods to establish Liouville-type theorems of elliptic equations, which was created by Alexandroff in the study of embedded constant mean curvature surfaces and then developed by Serrin \cite{Ser}, Gidas-Ni-Nirenberg \cite{GNN}, Caffarelli-Gidas-Spruck \cite{CGS}, Chen-Li \cite{CL1}, and C. Li \cite{Li}, among others. Its variant, the method of moving spheres, was also used and developed by Chen-Li \cite{CL3}, Li-Zhang \cite{LyZh}, Li-Zhu \cite{LyZ}, etc. Due to the nonlocal nature of the fractional Laplacian, it is difficult to apply the moving planes/spheres method to fractional order equations. As far as we know, there are three methods to carry out the moving planes/spheres for such nonlocal problems: extension method, integral equation method, and direct method. 

The extension method was based on the work of Caffarelli-Silvestre \cite{CaSi}. After being extended to one more dimension, a fractional order equation becomes a local problem. Then one can use the moving planes/spheres method to this local problem. The extension method has been successfully used to study fractional order equations, see, for example, Chen-Zhu \cite{CZ}, Fall-Weth \cite{FW1}, Jin-Li-Xiong \cite{JLX} and the references therein.

The moving planes method in integral forms was introduced by Chen, Li and Ou \cite{CLO3}, and the corresponding moving spheres method was developed by Li \cite{Ly}. By establishing the equivalence between fractional order equations and corresponding integral equations, one only needs to apply the moving planes/spheres method for the integral equations. The integral equation method can be seen in \cite{CDQ,CLO1,CL6,Y,ZLy,ZCCY} and the references therein.     

When using the extension method and the integral equation method, one sometimes needs to impose additional conditions on the solutions. Moreover, for uniformly elliptic nonlocal operators or fully nonlinear nonlocal operators (e.g., fractional $p$-Laplacian), it seems that neither the extension method nor the integral equation method works. This motivates Chen-Li-Li \cite{CLL} to come up with the direct method of moving planes for fractional Laplacian, where they established some interesting maximum principles for antisymmetric functions (see also Jin-Xiong \cite{JX} and Jarohs-Weth \cite{JW}). This direct method can also be applied to many nonlocal problems, see, for example, \cite{CHL,CL2,CLL1,LZ1,ZL} and the references therein. Later, Chen, Li and Zhang \cite{CLZ} developed the direct method of moving spheres in the fractional setting. They obtained a narrow region principle for spherically anti-symmetric functions and showed a Liouville-type theorem for the following fractional Laplacian equation
\begin{equation}\label{In-3}
  (-\Delta)^{\alpha/2} u(x)=f(u(x)),\,\,\,\, x\in \mathbb{R}^{n}.
\end{equation}  

However, the direct method of moving planes/spheres developed in the above articles under the non-local framework requires the assumption that $f$ is Lipschitz continuous or monotonically increasing. For instance, the method in Chen-Li-Li \cite{CLL} requires $f$ to be Lipschitz continuous, and the proof of Chen-Li-Zhang \cite[Theorem 2]{CLZ} essentially uses the monotonicity of the non-linearity $f$ (see \cite[Lemma A.2]{CLZ}). 
Moreover, the key estimates in both Chen-Li-Li \cite[Lemma A.2]{CLL} and Chen-Li-Zhang \cite[Lemma A.2]{CLZ} are obtained by the equivalent integral equation of \eqref{In-3}, and so an integrability condition on $f$ is also required to ensure the equivalence of \eqref{In-3} with its integral equation. In this paper, we will further develop the direct method of moving spheres to the fractional Laplacian equation \eqref{In-3} without relying on its integral equation. Our main idea is to construct elaborate anti-symmetric auxiliary functions such that the maximum principles of spherically anti-symmetric functions apply to a partial region rather than the whole outer region when moving spheres. Thus, the Lipschitz regularity and monotonicity of $f$ can be removed. We believe that this kind of idea and the method of constructing anti-symmetric functions can also be applied to a variety of problems involving nonlinear non-local operators and more general non-linearities.

We first prove the following pointwise estimates of $(-\Delta)^{\alpha/2} w$ at the minimum points for spherically anti-symmetric functions, which are inspired by Cheng-Huang-Li \cite{CHL}.   

\begin{theorem}[Maximum principle for spherically anti-symmetric functions]\label{MP 1}
 Let $w \in \mathcal{L}_\alpha$ be a spherically anti-symmetric function, i.e.,
 \begin{equation*}
      w(x)=-\frac{1}{|x|^{n-\alpha}} w\left(\frac{x}{|x|^2}\right), \quad \forall|x| \geq 1 \quad \text{(implying that $w(x)=0$ for $|x|=1$)}.
 \end{equation*}
 Denote $d(x)=\operatorname{dist}\left(x, \partial B_1(0)\right)$. 
\begin{itemize}
      \item[(1)] Suppose that there exists $|\bar{x}|>1$ satisfying 
\begin{equation*}
w(\bar{x})=\inf _{|x| \geq 1} w(x) \leq  0, 
\end{equation*}
and $w$ is $C^{1,1}_{loc}$ at $\bar{x}$. Then
\begin{equation*}
\begin{aligned}
(-\Delta)^{\alpha/2} w(\bar{x}) \leq C \Bigg(& \frac{ w(\bar{x})}{(d(\bar{x}))^\alpha(1+d(\bar{x}))^n} \\
& +\int_{|y|>1}\left(w(\Bar{x})-w(y)\right) \bigg( \frac{1}{|\Bar{x}-y|^{n+\alpha }}-\frac{1}{\left|\frac{\bar{x}}{|\Bar{x}|}-|\Bar{x}|y\right|^{n+\alpha}}\bigg)dy \Bigg). 
\end{aligned}
\end{equation*}
where $C=C(n, \alpha) >0$ is a constant. In particular, we have 
\begin{equation*}
    (-\Delta)^{\alpha/2} w(\bar{x}) \leq -C  \int_{|y|>1}w(y) \bigg( \frac{1}{|\Bar{x}-y|^{n+\alpha}}-\frac{1}{\left|\frac{\bar{x}}{|\Bar{x}|}-|\Bar{x}|y\right|^{n+\alpha}} \bigg) dy. 
\end{equation*}
\item[(2)] Suppose that there exists $0<|\bar{x}|<1$ satisfying 
\begin{equation*}
w(\bar{x})=\inf _{0<|x|<1} w(x) \leq  0, 
\end{equation*}
and $w$ is $C^{1,1}_{loc}$ at $\bar{x}$. Then
\begin{equation*}
\begin{aligned}
(-\Delta)^{\alpha/2} w(\bar{x}) \leq C \Bigg( & \frac{ w(\bar{x})}{(d(\bar{x}))^\alpha} \\
& + \int_{|y|<1}\left(w(\Bar{x})-w(y)\right) \bigg( \frac{1}{|\Bar{x}-y|^{n+\alpha }}-\frac{1}{\left|\frac{\bar{x}}{|\Bar{x}|}-|\Bar{x}|y\right|^{n+\alpha}}\bigg)dy \Bigg). 
\end{aligned}
\end{equation*}
where $C=C(n, \alpha) >0$ is a constant. 
 \end{itemize}
\end{theorem}

To establish a positive lower bound of the difference between the solution and its Kelvin transform near infinity, we also need the following maximum principle for fractional superharmonic functions.     

\begin{theorem}[Maximum principle for fractional  superharmonic functions]\label{sym mp}
Let $\lambda_{0}>0$ and $\Omega \subset B_{\lambda_{0}}^{c}(0)$ be an open set. Suppose that $h \in\mathcal{L}_{\alpha} \cap C_{loc}^{1,1}(\Omega)$ is lower semicontinuous on $\bar{\Omega}$ and satisfies 
$$
\begin{cases}
(-\Delta)^{\alpha / 2} h(x) \geq 0 ~~~ &\text { in } \Omega, \\
h(x) \geq 0 ~~~ & \text { in }  B_{\lambda_{0}}^{c}(0) \backslash \Omega, \\
h(x)=-\left(\frac{\lambda_{0}}{|x|}\right)^{n-\alpha} h\left(\frac{\lambda_{0}^{2} x}{|x|^{2}}\right) ~~~ &\text{ in } B_{\lambda_{0}}^{c}(0).  
\end{cases}
$$
If
\begin{equation*}
\liminf_{|x| \rightarrow \infty} h(x) \geq 0,  
\end{equation*}
then 
\begin{equation}\label{eq 20}
h(x) \geq 0 \;\;\;\; \text{ in } \Omega.    
\end{equation}
\end{theorem}

Based on Theorems \ref{MP 1} and \ref{sym mp}, we further develop the direct method of moving spheres for the non-local equation \eqref{In-3} and establish the following classification result. 
\begin{theorem}\label{main thm}
Suppose that $f:(0, +\infty) \rightarrow[0, +\infty)$ is locally bounded and $\frac{f(t)}{t^{\tau}}$ is monotonically decreasing in $(0, +\infty) $ with $\tau=\frac{n+\alpha}{n-\alpha}$. Let $u \in \mathcal{L}_{\alpha} \cap C_{\text {loc }}^{1,1}\left(\mathbb{R}^{n}\right)$ be a positive solution to
    \begin{equation}\label{main eq}
        (-\Delta)^{\alpha / 2} u(x)=f(u(x)), \quad\quad  x\in \mathbb{R}^{n}.
    \end{equation}
Then either $u\equiv C_0$ for some constant $C_0 > 0$ satisfying $f(C_0)=0$, or there exists $C>0$ such that $f(t)=Ct^\tau$ for every $t\in \Big( 0, \max\limits _{x \in \mathbb{R}^n} u(x) \Big]$ and
     \begin{equation*}
         u(x)=\frac{\beta_1}{\Big(\left|x-x_0\right|^2+\beta_2^2 \Big)^{\frac{n-\alpha}{2}}} \quad\quad  \forall x \in \mathbb{R}^n, 
     \end{equation*}
for some $\beta_1, \beta_2>0$ and $x_0 \in \mathbb{R}^n$.
\end{theorem}

A similar result was also obtained in Chen-Li-Zhang \cite[Theorem 2]{CLZ}. In the second step of moving spheres, they used a key estimate in \cite[Lemma A.2]{CLZ} whose proof relies on the integral equation 
\begin{equation}\label{int01}
u(x)=\int_{\mathbb{R}^n}\frac{u^p(y)}{|x-y|^{n-\alpha}}dy 
\end{equation} 
instead of
\begin{equation}\label{int02}
u(x)=\int_{\mathbb{R}^n}\frac{f(u(y))}{|x-y|^{n-\alpha}}dy. 
\end{equation} 
For the special case $f(t) = t^p$, the proof of \cite[Lemma A.2]{CLZ} takes advantage of the monotonically increasing property of $t^p$. Therefore, their method essentially uses the monotonicity of $f$. In addition, to ensure that any solution of the differential equation \eqref{main eq} satisfies the integral equation \eqref{int02}, an integrability condition on $f$ is also required. In our proof of Theorem \ref{main thm}, we fill these two gaps using Theorems \ref{MP 1} and \ref{sym mp}. By constructing appropriate anti-symmetric auxiliary functions, we will apply Theorem \ref{sym mp} to a partial region instead of the whole outer region. This effectively eliminates the monotonicity and integrability assumptions on $f$. Moreover, our proof of Theorem \ref{main thm} does not depend at all on the integral representation of the solution, and hence this method can also be used for more general non-local equations (where no equivalent integral equations exist).

As an application of Theorem \ref{main thm}, we use the direct moving sphere method to give an alternative proof for the Liouville-type theorem of the following fractional Lane-Emden equation in the half space  
\begin{equation}\label{half eq-In} 
    \begin{cases}
        (-\Delta)^{\alpha/2}u(x)=u^p(x),\;\;\;\; &x\in \mathbb{R}^n_+,\\
        u\equiv 0, \;\;\;\; &x\in (\mathbb{R}^n_+)^c,
    \end{cases}
\end{equation}
where $1<p\leq\frac{n+\alpha}{n-\alpha}$. 

\begin{theorem}\label{half thm} 
Let $1<p\leq\frac{n+\alpha}{n-\alpha}$ and $u\in \mathcal{L}_\alpha \cap C_{loc}^{1,1}(\mathbb{R}^n_+) \cap C_{loc}(\overline{\mathbb{R}^n_+})$ be a nonnegative solution to \eqref{half eq-In}.  Then $u\equiv 0$ in $\mathbb{R}^n_+$. 
\end{theorem}

The same result has been obtained in \cite{CFY,CLZ,CLL}, where the method is to first prove that the solution $u$ depends only on $x_n$ and then obtain the above Liouville-type theorem based on the integral equation equivalent to \eqref{half eq-In}. Our proof is inspired by Li-Lin \cite{LL} and does not depend on the corresponding integral equation. We first show the monotonicity of $u$ with respect to $x_n$ by the direct method of moving spheres, and then establish the Liouville-type theorem via the local blow up analysis. Because of the non-locality, a result in Du-Jin-Xiong-Yang \cite{DJXY} shows that the limiting equation reads as 
\[
 (-\Delta)^{\alpha/2}u_\infty=u_\infty^p+b \quad\quad  \text{in } \mathbb{R}^n 
\]
for some $b \geq 0$ when carrying out the local blow up analysis. This is different from the classical Laplacian case, and our Theorem \ref{main thm} for the general non-linearity will play an important role.

The paper is organized as follows. In Section \ref{pre}, we prove Theorems \ref{MP 1} and \ref{sym mp} and then apply them to establish several essential estimates. In Section \ref{main section}, we further develop the direct method of moving spheres to prove Theorem \ref{main thm}. In Section \ref{application}, we apply the direct method of moving spheres and local blow up argument to show Theorem \ref{half thm}.

\section{Maximum principles and essential estimates}\label{pre}
In this section, we prove the maximum principles in Theorems \ref{MP 1} and \ref{sym mp}. Then we apply Theorems \ref{MP 1} and \ref{sym mp} to establish several essential estimates that will be used in the next section.  
\begin{proof}[Proof of Theorem \ref{MP 1}]
(1) For convenience, we would omit $C_{n,\alpha} P.V.$ in the definition of $(-\Delta)^{\alpha/ 2}$. Notice that   
\begin{equation}\label{eq 10} 
    (-\Delta)^{\alpha/ 2} w(\bar{x})= \underbrace{\int_{|y|>1} \frac{w(\bar{x})-w(y)}{|\bar{x}-y|^{n+\alpha}} d y}_{\text{I}}+\underbrace{\int_{|y|<1} \frac{w(\bar{x})-w(y)}{|\bar{x}-y|^{n+\alpha}} d y}_{\text{II}}.
\end{equation}
First, we compute II: 
\begin{equation*}\label{eq 5}
\begin{aligned}
        \text{II}=&\int_{|z|>1} \frac{w(\bar{x})-w\left(\frac{z}{|z|^{2}}\right)}{\left\lvert\, \bar{x}-\frac{z}{|z|^{2}}\right\lvert^{n+\alpha}} \frac{d z}{|z|^{2 n}}
        =\int_{|z|>1} \frac{w(\bar{x})+|z|^{n-\alpha}w(z)}{|\bar{x}|^{n+\alpha}|z|^{n-\alpha}\left|\frac{\bar{x}}{|\bar{x}|^{2}}-z\right|^{n+\alpha}}d z\\
        =&\int_{|z|>1} \frac{w(\bar{x})(1+|z|^{n-\alpha})+|z|^{n-\alpha}(w(z)-w(\Bar{x}))}{|\bar{x}|^{n+\alpha}|z|^{n-\alpha}\left|\frac{\bar{x}}{|\bar{x}|^{2}}-z\right|^{n+\alpha}}dz\\
        =&\underbrace{w(\bar{x})\int_{|z|>1} \frac{1+|z|^{n-\alpha}}{|\bar{x}|^{n+\alpha}\left|\frac{\bar{x}}{|\bar{x}|^{2}}-z\right|^{n+\alpha}|z|^{n-\alpha}} d z}_{\text{II}^\prime}-\underbrace{ \int_{|z|>1} \frac{w(\bar{x})-w(z)}{\left|\frac{\bar{x}}{|\bar{x}|}-|\bar{x}|z\right|^{n+\alpha}} d z}_{\text{II}^{\prime\prime}}.
\end{aligned}
\end{equation*}
Now we estimate $\text{II}^{\prime}$. Since $w(\bar{x}) \leq 0$, it follows that
\begin{equation}\label{eq 2}
\begin{aligned}
\text{II}^{\prime}
& =\frac{w (\bar{x})}{|\bar{x}|^{n+\alpha}} \int_{|z|>1} \frac{\frac{1}{|z|^{n-\alpha}}+1}{\left|\frac{\bar{x}}{|\bar{x}|^2}-z\right|^{n+\alpha}} d z =w(\bar{x})\int_{|y|<1} \frac{1+\frac{1}{|y|^{n-\alpha}}}{|\bar{x}-y|^{n+\alpha}} d y  \leq  w(\bar{x}) \int_{|y|<1} \frac{d y}{|\bar{x}-y|^{n+\alpha}}.
\end{aligned}
\end{equation}
For $1<|\bar{x}|<3$, we have 
\begin{equation}\label{eq 35}
     w(\bar{x}) \int_{|y|<1} \frac{d y}{|\bar{x}-y|^{n+\alpha}} \leq C w(\bar{x}) \int_{d(\bar{x})}^{d(\bar{x})+1} \frac{r^{n-1}}{r^{n+\alpha}} d r \leq  \frac{C w(\bar{x})}{(d(\bar{x}))^\alpha}, 
\end{equation}
where $C=C(n,\alpha)>0$ and $d(\bar{x})=\operatorname{dist}\left(\bar{x}, \partial B_1(0)\right)$. For $|\bar{x}|\geq3$, we have 
\begin{equation}\label{eq 36}
     w(\bar{x}) \int_{|y|<1} \frac{d y}{|\bar{x}-y|^{n+\alpha}} \leq   C w(\bar{x}) \int_{|y|<1} \frac{d y}{|\bar{x}|^{n+\alpha}} \leq \frac{C w(\bar{x})}{(d(\bar{x}))^{n+\alpha}}.
\end{equation}
Combining (\ref{eq 2}), (\ref{eq 35}) and (\ref{eq 36}), we obtain
\begin{equation*}
    \text{II}^{\prime}\leq\frac{C  w(\bar{x})}{(d(\bar{x}))^{\alpha}(1+d(\bar{x}))^{n}}.
\end{equation*}
Thus
\begin{equation*}
    \begin{aligned}
         (-\Delta)^{\alpha/ 2} w(\bar{x}) & =\text{II}^\prime+(\text{I}+\text{II}^{\prime\prime}) \\
         & \leq C \Bigg( \frac{ w(\bar{x})}{(d(\bar{x}))^{\alpha}(1+d(\bar{x}))^{n}} \\
         & ~~~~~ +\int_{|y|>1}\left(w(\Bar{x})-w(y)\right)\bigg(\frac{1}{|\Bar{x}-y|^{n+\alpha}}-\frac{1}{\left|\frac{\bar{x}}{|\Bar{x}|}-|\Bar{x}|y\right|^{n+\alpha}}\bigg) dy \Bigg). 
    \end{aligned} 
\end{equation*}
Moreover, since $\left|\frac{x}{|x|}-|x|y\right|^{2}- |x-y|^{2} =(|x|^2 -1)(|y|^2 -1) >0 $ for $|x|, |y|>1$, we have
\begin{equation*}\label{eq 3}
 (-\Delta)^{\alpha/ 2} w(\bar{x}) \leq -C \int_{|y|>1}w(y)\bigg(\frac{1}{|\Bar{x}-y|^{n+\alpha}}-\frac{1}{\left|\frac{\bar{x}}{|\Bar{x}|}-|\Bar{x}|y\right|^{n+\alpha}}\bigg)dy.
\end{equation*}

(2) The proof in the second part is similar to that of (1), so we only give the sketch. Let $0<|\bar{x}|<1$ satisfy 
\begin{equation*}
w(\bar{x})=\inf _{0<|x|<1} w(x) \leq  0. 
\end{equation*} 
By the definition of $(-\Delta)^{\alpha/ 2}$ and the spherical anti-symmetry, we have
\begin{equation*} 
\begin{aligned}
(-\Delta)^{\alpha/ 2} w(\bar{x}) &= \int_{|y|>1} \frac{w(\bar{x})-w(y)}{|\bar{x}-y|^{n+\alpha}} d y + \int_{|y|<1} \frac{w(\bar{x})-w(y)}{|\bar{x}-y|^{n+\alpha}} d y  \\
& =  \int_{|y|<1} \frac{w(\bar{x}) + |y|^{n-\alpha} w(y)}{\left| \frac{\bar{x}}{|\bar{x}|} - |\bar{x}| y \right|^{n+\alpha}}\frac{1}{|y|^{n-\alpha}}  dy + \int_{|y|<1} \frac{w(\bar{x})-w(y)}{|\bar{x}-y|^{n+\alpha}} d y \\
&= w(\bar{x}) \int_{|y|<1} \frac{1+|y|^{-(n-\alpha)}}{\left| \frac{\bar{x}}{|\bar{x}|} - |\bar{x}| y \right|^{n+\alpha}} dy  \\
& ~~~~~ + \int_{|y|<1} (w(\bar{x})-w(y)) \bigg( \frac{1}{|\bar{x}-y|^{n+\alpha}} - \frac{1}{\left| \frac{\bar{x}}{|\bar{x}|} - |\bar{x}| y \right|^{n+\alpha}} \bigg) dy,
\end{aligned}
\end{equation*} 
where we have also used the fact $\left|\bar{x} - \frac{y}{|y|^2} \right| |y|= \left| \frac{\bar{x}}{|\bar{x}|^2} -  y \right| |\bar{x}|$ in the second identity. Notice that for $0 < |\bar{x}|< 1$,  
\begin{equation*} 
\begin{aligned}
\int_{|y|<1} \frac{1+|y|^{-(n-\alpha)}}{\left| \frac{\bar{x}}{|\bar{x}|} - |\bar{x}| y \right|^{n+\alpha}} dy &= \frac{1}{|\bar{x}|^{n+\alpha}} \int_{|y|<1} \frac{1+|y|^{-(n-\alpha)}}{\left| \frac{\bar{x}}{|\bar{x}|^2} -  y \right|^{n+\alpha}} dy \\
&  =   \int_{|y| > 1} \frac{1+|y|^{-(n-\alpha)}}{\left|\bar{x} -  y \right|^{n+\alpha}} dy \\
& \geq \frac{C}{(d(\bar{x}))^\alpha}, 
\end{aligned}
\end{equation*} 
where $C=C(n,\alpha)>0$ and $d(\bar{x})=\operatorname{dist}\left(\bar{x}, \partial B_1(0)\right)$. Hence, we obtain   
\begin{equation*}
\begin{aligned}
(-\Delta)^{\alpha/2} w(\bar{x}) \leq C \Bigg( & \frac{ w(\bar{x})}{(d(\bar{x}))^\alpha} +\int_{|y|<1}\left(w(\Bar{x})-w(y)\right) \bigg( \frac{1}{|\Bar{x}-y|^{n+\alpha }}-\frac{1}{\left|\frac{\bar{x}}{|\Bar{x}|}-|\Bar{x}|y\right|^{n+\alpha}}\bigg)dy \Bigg). 
\end{aligned}
\end{equation*}
This completes the proof of Theorem \ref{MP 1}. 
\end{proof}

Now we prove Theorem \ref{sym mp}.
\begin{proof}[Proof of Theorem \ref{sym mp}]
    Suppose that (\ref{eq 20}) is not true. Then there exists $ x^{0} \in \Omega$ such that
$$
h\left(x^{0}\right)=\min _{x\in\Omega} h(x)<0 .
$$
By the definition, we have
\begin{equation}\label{eq 21}
(-\Delta)^{\alpha / 2} h(x^{0})=C_{n, \alpha}\Big(\int_{B_{\lambda_{0}}^{c}(0)} + P.V.\int_{B_{\lambda_{0}}(0)}\Big)\Big(\frac{h\left(x^{0}\right)-h(y)}{\left|x^{0}-y\right|^{n+\alpha}} d y\Big) .
\end{equation}
First, we compute the second term in (\ref{eq 21}). By spherical anti-symmetry,
\begin{equation}\label{eq 22}
\begin{aligned}
    &  C_{n, \alpha} P.V.\int_{B_{\lambda_{0}}(0)}\frac{h\left(x^{0}\right)-h(y)}{\left|x^{0}-y\right|^{n+\alpha}} d y \\
    & =  C_{n, \alpha} P.V.\int_{B_{\lambda_{0}}(0)}\frac{h\left(x^{0}\right)+\left(\frac{\lambda_{0}}{|y|}\right)^{n-\alpha} h\left(\frac{\lambda_{0}^{2} y}{|y|^{2}}\right)}{\left|x^{0}-y\right|^{n+\alpha}}dy\\
     & =  C_{n, \alpha} P.V.\int_{{|z|>\lambda_0}}\frac{h\left(x^{0}\right)+\left(\frac{|z|}{\lambda_{0}}\right)^{n-\alpha} h(z)}{\left|x^{0}-\frac{\lambda_0^2}{|z|^2}z\right|^{n+\alpha}}\Big(\frac{\lambda_0}{|z|}\Big)^{2n} dz\\
    &  =h\left(x^{0}\right) \int_{|z|>\lambda_{0}} \frac{1+\left(\frac{|z|}{\lambda_{0}}\right)^{n-\alpha}}{\left|x^{0}-\frac{\lambda_0^2}{|z|^2}z\right|^{n+\alpha}} \Big(\frac{\lambda_0}{|z|}\Big)^{2n} d z-\int_{|z|>\lambda_{0}} \frac{h\left(x^{0}\right)-h(z)}{\left|\frac{|x^{0}|}{\lambda_0}z-\frac{\lambda_0}{|x^0|}x^0\right|^{n+\alpha}} d z.
\end{aligned}
\end{equation}
Combining (\ref{eq 21}) and (\ref{eq 22}), we have
$$
\begin{aligned}
(-\Delta)^{\alpha / 2} h\left(x^{0}\right) = ~& C_{n, \alpha}\int_{B_{\lambda_{0}}^{c}(0)} \Big(\frac{1}{\left|x^{0}-y\right|^{n+\alpha}}-\frac{1}{\left|\frac{\left|x^{0}\right|}{\lambda_{0}} y-\frac{\lambda_{0}}{ |x_{0}|}x^{0}\right|^{n+\alpha}}\Big)\left(h\left(x^{0}\right)-h(y)\right) d y \\
& + C_{n, \alpha} h\left(x^{0}\right) \int_{|z|>\lambda_{0}} \frac{1+\left(\frac{|z|}{\lambda_{0}}\right)^{n-\alpha}}{\left|x^{0}-\frac{\lambda_0^2}{|z|^2}z\right|^{n+\alpha}} \left(\frac{\lambda_0}{|z|}\right)^{2n} d z.
\end{aligned} 
$$
Note that $h(x^0)<0$, $h\left(x^{0}\right)-h(y)<0$ for $|y|>\lambda_0$, and
\begin{equation*}
    \left| \frac{|x|}{\lambda_{0}} y-\frac{\lambda_{0}}{|x|}x\right|^2 > \left|x-y\right|^2\;\;\;\;\text{ for }|x|,|y|>\lambda_0.
\end{equation*}
Thus $(-\Delta)^{\alpha / 2} h\left(x^{0}\right) <0$. This is a contradiction with the assumption.  
\end{proof}

Next, we apply Theorems \ref{MP 1} and \ref{sym mp} to establish some essential estimates that will be used in Section \ref{main section}. By a B\v{o}cher type theorem, we give an estimate for the solution $u$ of \eqref{main eq} near infinity.  
\begin{lemma}[Estimate of the solution at infinity]\label{u at infty}
 Suppose that $u$ satisfies all the assumptions in Theorem \ref{main thm}.
Then there exists a positive constant $C$ such that
\begin{equation*}
 u(x) \geq \frac{C}{|x|^{n-\alpha}} \quad\quad\quad  \text{for all} ~ |x|>1.
\end{equation*}
\end{lemma}

\begin{proof}
Consider the Kelvin transform of $u$
\begin{equation*}
    u_\lambda(x)=\big(\frac{\lambda}{|x|}\big)^{n-\alpha}u\big(\frac{\lambda^2}{|x|^2}x\big)\quad x\neq 0.
\end{equation*}
Then $u_{\lambda}(x) \geq 0$ for $x\neq 0$, and $(-\Delta)^{\alpha / 2} u_{\lambda}(x)=\left(\frac{\lambda}{|x|}\right)^{n+\alpha} f\left(u\left(\frac{\lambda^{2}}{|x|^{2}} x\right)\right)\geq 0$. Denote
\begin{equation*}
    m=\inf_{B_r(0) \backslash B_{\frac{r}{2}}(0)}u_\lambda>0
\end{equation*}
for some fixed $r\in(0,1]$. By \cite[Theorem 1]{LWX}, there exists a positive constant $c=c(n,\alpha)<1$ such that
\begin{equation*}
u_{\lambda}(x) \geq c m,\quad \forall 0<|x|<r.
\end{equation*}
Thus
\begin{equation*}
   u(x) \geq \frac{c m \lambda^{n-\alpha}}{|x|^{n-\alpha}}, \quad|x|>\frac{\lambda^{2}}{r}.
\end{equation*}
One can get the conclusion by setting $\lambda=1$.
\end{proof}

The next lemma is a part of the moving spheres method.
\begin{lemma}\label{step 2.1}
 Assume that $f:(0, +\infty) \rightarrow[0,+\infty)$ is locally bounded, and $\frac{f(t)}{t^{\tau}}$ is 
 monotonically decreasing in $(0, +\infty)$ with $\tau=\frac{n+\alpha}{n-\alpha}$. Let $u \in \mathcal{L}_{\alpha} \cap C_{\text {loc}}^{1,1}\left(\mathbb{R}^{n}\right)$ is a positive solution to
\begin{equation*}
(-\Delta)^{\alpha / 2} u(x)=f(u(x)) \quad x \in \mathbb{R}^{n}.
\end{equation*}
Denote the Kelvin transform of $u$ as
\begin{equation*}
u_{\lambda}(x)=\left(\frac{\lambda}{\left|x-x_{0}\right|}\right)^{n-\alpha} u\left(\frac{\lambda^{2}}{\left|x-x_{0}\right|^{2}}\left(x-x_{0}\right)+x_{0}\right), \quad x \neq x_{0}.
\end{equation*}
If $u(x) \geq u_{\lambda}(x)$ for all $\left|x-x_{0}\right|>\lambda$, 
then either
\begin{equation*}
u(x)>u_{\lambda}(x), \quad \forall  \left|x-x_{0}\right|>\lambda,
\end{equation*}
or else 
\begin{equation*}
u(x) \equiv u_{\lambda}(x), \quad \forall  \left|x-x_{0}\right|>\lambda .
\end{equation*}
\end{lemma}

\begin{proof}
Without loss of generality, we assume $x_0=0$ in the Kelvin transform. Let $w_{\lambda}=u-u_{\lambda}$. Then $w_{\lambda}(x) \geq 0$ for any $|x| \geq \lambda$. If $w_{\lambda} \equiv 0$ for any $|x| \geq \lambda$, then the first statement holds. Next we consider $w_{\lambda} \not \equiv 0$. We will prove that, in this case,
    \begin{equation}\label{eq 11}
w_{\lambda}>0, \quad \forall |x| \geq \lambda
\end{equation}
We prove (\ref{eq 11}) by contradiction. Suppose that there exists $|\bar{x}|>\lambda$ such that
\begin{equation*}
w_{\lambda}(\bar{x})=0.
\end{equation*}
One can see that
\begin{equation}\label{eq 12}
\begin{aligned}
(-\Delta)^{\alpha / 2} w_{\lambda}(\bar{x})
& =f(u(\bar{x}))-\Big(\frac{\lambda}{|\bar{x}|}\Big)^{n+\alpha} f\left(\Big(\frac{|\bar{x}|}{\lambda}\Big)^{n-\alpha} u_{\lambda}(\bar{x})\right) \\
& =f(u(\bar{x}))-\Big(\frac{\lambda}{|\bar{x}|}\Big)^{n+\alpha} f\left(\Big(\frac{|\bar{x}|}{\lambda}\Big)^{n-\alpha} u(\bar{x})\right) \\
&=\frac{f(u(\bar{x}))}{u^{\tau}(\bar{x})} u^{\tau}(\bar{x})-\frac{f\left(\Big(\frac{|\bar{x}|}{\lambda}\Big)^{n-\alpha} u(\bar{x})\right)}{\left(\Big(\frac{|\bar{x}|}{\lambda}\Big)^{n-\alpha} u(\bar{x})\right)^{\tau}} u^{\tau}(\bar{x}) \\
& \geq\frac{f(u(\bar{x}))}{u^{\tau}(\bar{x})} u^{\tau}(\bar{x})-\frac{f(u(\bar{x}))}{u^{\tau}(\bar{x})} u^{\tau}(\bar{x})=0.
\end{aligned}
\end{equation}
Here the last inequality is due to $\left(\frac{|\bar{x}|}{\lambda}\right)^{n-\alpha} u(\bar{x})>u(\bar{x})$ and
the monotonicity of $\frac{f(t)}{t^{\tau}}$. Let $w(x)=w_{\lambda}\left(\lambda x\right)$. Then $w(x) \geq 0$ for any $|x| \geq 1$, and there exists a point (still denote as $\bar{x}$) such that $|\bar{x}|>1$ and $w(\bar{x})=0$. By Theorem \ref{MP 1} and the fact that $\left|\frac{x}{|x|}-|x|y\right|^{2}>|x-y|^{2}$ for $|x|,|y|>1$, we have
\begin{equation}\label{eq 14}
 (-\Delta)^{\alpha/ 2} w(\bar{x}) \leq -\int_{|y|>1} w(y) \Bigg(\frac{1}{|\Bar{x}-y|^{n+\alpha}}-\frac{1}{\left|\frac{\bar{x}}{|\Bar{x}|}-|\Bar{x}|y\right|^{n+\alpha}}\Bigg)dy\leq 0
\end{equation}
Combining (\ref{eq 12}) and (\ref{eq 14}),  we get
\begin{equation*}
\int_{|y|>1} w(y) \Bigg(\frac{1}{|\Bar{x}-y|^{n+\alpha}}-\frac{1}{\left|\frac{\bar{x}}{|\Bar{x}|}-|\Bar{x}|y\right|^{n+\alpha}}\Bigg)dy =0 .
\end{equation*}
Thus
\begin{equation*}
w(y) \equiv 0, \quad|y|>1,
\end{equation*}
which contradicts (\ref{eq 11}).  
\end{proof}

We construct a spherically anti-symmetric auxiliary function, which will be useful to estimate the difference between the solution and its Kelvin transform near infinity.  

\begin{lemma}[Spherically anti-symmetric auxiliary function]\label{aux fun}
    Let $u$ and $u_\lambda$ be defined as in Lemma \ref{step 2.1}. Without loss of generality we assume $x_0=0$ in the Kelvin transform. Let $w_\lambda =u-u_\lambda$ and
    \begin{equation*}
    \lambda_{0} =\sup \{\mu>0\mid u_{\lambda}(x) \leq  u(x), ~~ \forall ~ 0<\lambda<\mu, ~~|x| \geq \lambda \} < \infty.
\end{equation*}
For $\varepsilon>0$ and $R \gg \lambda_{0}+1$, define 
\begin{equation*}
\varphi(x)= \begin{cases}
\varepsilon /|x|^{n-\alpha}, & |x|>R, \\ 
\omega_{\lambda_{0}}(x), & \lambda_{0} \leq|x| \leq R, \\
-\left(\frac{\lambda_{0}}{|x|}\right)^{n-\alpha} \varphi\left(\frac{\lambda_{0}^{2} x}{|x|^{2}}\right), \quad & 0<|x|<\lambda_{0}.
\end{cases}
\end{equation*}
Then for $\varepsilon>0$ sufficiently small, $\varphi$ is upper semicontinuous and satisfies 
\begin{equation*}
(-\Delta)^{\alpha / 2} \varphi(x) \leq 0, \quad x \in B_{2R}^{c}(0).
\end{equation*}
\end{lemma}

\begin{proof}
By the spherical anti-symmetry, we have that for any $x \in B_{2 R}^{c}$, 
\begin{align}\label{eq 23}
& (-\Delta)^{\alpha / 2} \varphi(x) \\
& =C_{n, \alpha} \int_{|y|>\lambda_{0}} \frac{\varphi(x)-\varphi(y)}{|x-y|^{n+\alpha}}+\frac{\left(\frac{\lambda_{0}}{|y|}\right)^{n-\alpha} \varphi(x)+\varphi(y)}{\left|\frac{\lambda_{0}}{|x|} x-\frac{|x|}{\lambda_{0}} y\right|^{n+\alpha}} d y \nonumber\\
& =C_{n, \alpha} \int_{|y|>\lambda_{0}}  K(x, y) (\varphi(x)-\varphi(y)) d y+\varphi(x) \int_{|y|>\lambda_{0}} \frac{1+\left( \frac{\lambda_{0}}{ |y|}\right)^{n-\alpha}}{\left|\frac{\lambda_{0}}{|x|} x-\frac{|x|}{\lambda_{0}} y\right|^{n+\alpha}} d y \nonumber\\
& =C_{n, \alpha} \int_{\lambda_{0}<|y| \leq R}K(x,y)\left(\frac{\varepsilon}{|x|^{n-\alpha}}-\omega_{\lambda_{0}}(y)\right) d y  + C_{n, \alpha} \int_{|y|>R} K(x,y)\left(\frac{\varepsilon }{|x|^{n-\alpha}}-\frac{\varepsilon }{|y|^{n-\alpha}}\right) d y \nonumber\\
& ~~~~~ + \frac{\varepsilon C_{n, \alpha}}{|x|^{n-\alpha}} \int_{|y|>\lambda_{0}} \frac{1+\left(\frac{\lambda_{0}}{|y|}\right)^{n-\alpha}}{\left|\frac{ \lambda_{0}}{|x|} x-\frac{|x|}{\lambda_{0}} y\right|^{n+\alpha}} d y \nonumber ,
\end{align}
where
\begin{equation*}
 K(x, y)=\frac{1}{|x-y|^{n+\alpha}}-\frac{1}{\left|\frac{\lambda_{0}}{|x|} x-\frac{|x|}{\lambda_{0}} y\right|^{n+\alpha}}.
\end{equation*}
One can see that $K(x, y)>0$ for $|x|,|y|>\lambda_{0}$. On the other hand, since $\frac{\varepsilon}{|x|^{n-\alpha}}$ is a fundamental solution to $(-\Delta)^{\alpha / 2}$, we have

\begin{align*}
0 & =(-\Delta)^{\alpha / 2} \frac{\varepsilon}{|x|^{n-\alpha}}\\
&=\varepsilon C_{n,\alpha}\Bigg(\int_{|y|>\lambda_{0}} \frac{1 /|x|^{n-\alpha}-1 /|y|^{n-\alpha}}{|x-y|^{n+\alpha}} d y+\int_{|y|<\lambda_{0}} \frac{1 /|x|^{n-\alpha}-1 /|y|^{n-\alpha}}{|x-y|^{n+\alpha}} d y\Bigg) \\
& =\varepsilon C_{n,\alpha}\Bigg(\int_{|y|>\lambda_{0}} \frac{1 /|x|^{n-\alpha}-1 /|y|^{n-\alpha}}{|x-y|^{n+\alpha}} d y+\int_{|y|>\lambda_{0}} \frac{\left(\frac{\lambda_{0}}{|x| |y|} \right)^{n-\alpha}-1 / \lambda_{0}^ {n-\alpha}}{\left|\frac{\lambda_{0}}{|x|} x-\frac{|x|}{\lambda_{0}} y\right|^{n+\alpha}} d y\Bigg) \\
& =\varepsilon C_{n, \alpha}\Bigg(\int_{|y|>\lambda_{0}}  K(x, y)\left(\frac{1}{|x|^{n-\alpha}}-\frac{1}{|y|^{n-\alpha}}\right)dy + \frac{1}{|x|^{n-\alpha}} \int_{|y|>\lambda_{0}} \frac{1+\left(\frac{\lambda_0}{|y|}\right)^{n-\alpha}}{\left|\frac{\lambda_{0}}{|x|} x-\frac{|x|}{\lambda_{0}} y\right|^{n+\alpha}} dy \\
& ~~~~~ -\int_{|y|>\lambda_{0}} \frac{\frac{1}{|y|^{n-\alpha}}+\frac{1}{\lambda_0^{n-\alpha}}}{\left|\frac{\lambda_{0}}{|x|} x-\frac{|x|}{\lambda_{0}} y\right|^{n+\alpha}}dy\Bigg).
\end{align*}
Thus we obtain 
\begin{equation}\label{eq 24}
\frac{1}{|x|^{n-\alpha}} \int_{|y|>\lambda_{0}} \frac{1+\left(\frac{\lambda_{0}}{|y|}\right)^{n-\alpha}}{\left|\frac{\lambda_{0}}{|x|} x-\frac{|x|}{\lambda_{0}} y\right|^{n+\alpha}} d y=\int_{|y|>\lambda_{0}} \frac{\frac{1}{|y|^{n-\alpha}}+\frac{1}{\lambda_{0}^{n-\alpha}}}{\left|\frac{\lambda_{0}}{|x|} x-\frac{|x|}{\lambda_{0}} y\right|^{n+\alpha}} d y -\int_{|y|>\lambda_{0}} K(x, y)\left(\frac{1}{|x|^{n-\alpha}}-\frac{1}{|y|^{n-\alpha}}\right) d y .
\end{equation}
Using (\ref{eq 24}) to replace the last term in (\ref{eq 23}), we have
\begin{align*}
&(-\Delta)^{\alpha/ 2} \varphi(x)\\
&=C_{n, \alpha} \int_{\lambda_{0}<|y| \leq R} K(x, y)\left(\frac{\varepsilon}{|x|^{n-\alpha}}-\omega_{\lambda_{0}}(y)\right) d y
+\varepsilon C_{n, \alpha} \int_{|y|>R} K(x, y)\left(\frac{1}{|x|^{n-\alpha}}-\frac{1}{|y|^{n-\alpha}}\right) d y  \\
& ~~~~~ +\varepsilon C_{n, \alpha}\int_{|y|>\lambda_{0}} \frac{\frac{1}{|y|^{n-\alpha}}+\frac{1}{\lambda_0^{n-\alpha}}}{\left|\frac{\lambda_{0}}{|x|} x-\frac{|x|}{\lambda_{0}} y\right|^{n+\alpha}}dy
-\varepsilon C_{n ,\alpha} \int_{|y|>\lambda_{0}} K(x, y)\left(\frac{1}{|x|^{n-\alpha}}-\frac{1}{|y|^{n-\alpha}}\right) d y \\
& =C_{n, \alpha} \int_{\lambda_{0}<|y| \leq R} K(x, y)\left(\frac{\varepsilon}{|y|^{n-\alpha}}-w_{\lambda_{0}}(y)\right) d y
+\varepsilon C_{n, \alpha}\int_{|y|>\lambda_{0}} \frac{\frac{1}{|y|^{n-\alpha}}+\frac{1}{\lambda_0^{n-\alpha}}}{\left|\frac{\lambda_{0}}{|x|} x-\frac{|x|}{\lambda_{0}} y\right|^{n+\alpha}}dy\\
& =: \text{I}+\text{II}.
\end{align*}
First we estimate II: 
\begin{equation*}
\begin{aligned}
& \text { II } \leq  \frac{\varepsilon C_{n, \alpha, \lambda_{0}}}{|x|^{n+\alpha}} \int_{|y|>\lambda_{0}} \frac{d y}{\left|\frac{\lambda_{0}^{2}}{|x|^{2}} x-y\right|^{n+\alpha}}
\leq\frac{ \varepsilon C_{n, \alpha, \lambda_{0}} }{|x|^{n+\alpha}} \int_{|y|>\lambda_{0}} \frac{d y}{|y|^{n+\alpha}}
\leq \frac{\varepsilon C_{n, \alpha, \lambda_{0}}}{|x|^{n+\alpha}}.
\end{aligned}
\end{equation*}
Now we estimate I. By the mean value theorem,
\begin{equation*}
\begin{aligned}
K(x, y) =\frac{C_{n, \alpha}}{\xi^{ \frac{n+2}{2} +1}}\Bigg(\left|\frac{\lambda_{0}}{|x|} x-\frac{|x|}{\lambda_{0}} y\right|^{2}-|x-y|^{2}\Bigg) =\frac{C_{n, \alpha}}{\xi ^{\frac{n+2}{2}+1}} \frac{1}{\lambda_{0}^{2}}\left(|x|^{2}-\lambda_{0}^{2}\right)\left(|y|^{2}-\lambda_{0}^{2}\right),
\end{aligned}
\end{equation*}
where
\begin{equation*}
    |x-y|^{n+\alpha+2}<\xi^{ \frac{n+2}{2} +1}<\left|\frac{\lambda_{0}}{|x|} x-\frac{|x|}{\lambda_{0}} y\right|^{n+\alpha+2}.
\end{equation*}
Denote $M=\min\limits_{\lambda_0+1\leq |y|\leq R}w_{\lambda_0}(y)>0$. Choose a sufficiently small $\varepsilon$ such that $\max\limits_{\lambda_0+1\leq |y|\leq R} \frac{\varepsilon}{|y|^{n-\alpha}}<\frac{M}{2}$. Then
\begin{align*}
\text{I}& =C_{n, \alpha} \bigg(\int_{\lambda_{0}+1 \leq|y| \leq R} + \int_{\lambda_{0} \leq|y|<\lambda_{0}+1} \bigg) K(x, y) \bigg(\frac{\varepsilon}{|y|^{n-\alpha}}-w_{\lambda_{0}}(y)\bigg) d y \\
&\leq-\frac{M}{2} C_{n, \alpha} \int_{\lambda_{0}+1 \leq|y| \leq R} K(x, y) d y+C_{n, \alpha} \int_{\lambda_{0} \leq|y|<\lambda_{0}+1} \frac{1}{|x-y|^{n+2}} \cdot \frac{\varepsilon}{|y|^{n-\alpha}} d y\\
&\leq-M C_{n, \alpha, \lambda_{0}} \int_{\lambda_{0}+1 \leq|y| \leq R} \frac{\left(|x|^{2}-\lambda_{0}^{2}\right)\left(|y|^{2}-\lambda_{0}^{2}\right)}{\left|\frac{\lambda_{0}}{|x|} x-\frac{|x|}{\lambda_{0}} y\right|^{n+\alpha+2}} d y+C_{n, \alpha, \lambda_0} \frac{\varepsilon}{|x|^{n+\alpha}} \int_{\lambda_{0} \leq|y|<\lambda_{0}+1} \frac{d y}{|y|^{n-\alpha}}\\
&\leq-\frac{M C_{n, \alpha, \lambda_{0}}}{|x|^{n+\alpha}} \int_{\lambda_{0}+1 \leq|y| \leq R} \frac{|y|^{2}-\lambda_{0}^{2}}{\left(|y|+\lambda_{0}\right)^{n+\alpha+2}} d y+\frac{\varepsilon C_{n, \alpha, \lambda_{0}}}{|x|^{n+\alpha}} 
\leq-\frac{M C_{n, \alpha,\lambda_0}}{|x|^{n+\alpha}}+\frac{\varepsilon C_{n, \alpha, \lambda_{0}}}{|x|^{n+\alpha}}.
\end{align*}
Hence, we have for $x \in B_{2 R}^{c}$, 
\begin{equation*}
    (-\Delta)^{\alpha/ 2} \varphi(x) \leq \text{I}+\text{II} \leq-\frac{M C_{n, \alpha, \lambda_0}}{|x|^{n+\alpha}}+\frac{\varepsilon C_{n, \alpha,\lambda_{0}}}{|x|^{n+\alpha}},
\end{equation*}
 where $M=\min\limits _{\lambda_{0}+1 \leq|y| \leq R} w_{\lambda_{0}}(y)$.
By choosing $\varepsilon >0 $ small enough, we obtain that 
\begin{equation*}
    (-\Delta)^{\alpha / 2} \varphi(x) \leq 0 \;\;\;\;\text{ for any }x \in B_{2 R}^{c}.
\end{equation*}
The proof of Lemma \ref{aux fun} is finished.  
\end{proof}

Comparing $\varphi$ in Lemma \ref{aux fun} with $w_{\lambda_0}=u - u_{\lambda_0}$, we derive the asymptotic behavior of $w_{\lambda_0}$ at infinity.

\begin{lemma}[Estimate of $w_{\lambda_0}$ at infinity]\label{estimates w}
    Let $u$, $u_\lambda$, $w_\lambda$ and $\lambda_0$ be defined as in Lemma \ref{aux fun}. Without loss of generality we assume $x_0=0$ in the Kelvin transform. Assume that
    \begin{equation*}
        u(x)>u_{\lambda_0}(x) \;\;\;\;\text{ for any }|x|\geq \lambda_0.
    \end{equation*}
Then for $R \gg \max\{\lambda_0+1,2 \lambda_{0}\}$, there exists $\varepsilon>0$ such that
$$
w_{\lambda_{0}}(x) \geq \frac{\varepsilon}{|x|^{n-\alpha}} \;\;\;\;\text{ for any }|x|>2 R.
$$
\end{lemma}

\begin{proof}
Denote 
\[D=\left\{x>R: u_{\lambda_0}(x)<u(x)<2 u_{\lambda_{0}}(x)\right\} 
\]
and 
\[
\tilde{D}=\left\{x>R: u(x) \geq 2 u_{\lambda_{0}}(x)\right\}.
\]
For $x \in D$, one can see that
$$
\begin{aligned}
(-\Delta)^{\alpha / 2} w_{\lambda_{0}}(x) &=f(u(x))-\Big(\frac{\lambda_{0}}{|x|}\Big)^{n+\alpha} f\Big(\big(\frac{|x|}{\lambda_{0}}\big)^{n-\alpha} u_{\lambda_{0}}(x)\Big)  \\
& =\frac{f(u(x))}{u^{\tau}(x)} u^{\tau}(x)-\frac{f\left(\big(\frac{|x|}{\lambda_{0}}\big)^{n-\alpha} u_{\lambda_{0}}(x)\right)}{\left(\big(\frac{| x|}{\lambda_{0}}\big)^{n-\alpha}  u_{\lambda_{0}}(x)\right)^{\tau}} u_{\lambda_{0}}^{\tau}(x) \\
& \geq \frac{f(u(x))}{u^{\tau}(x)}\left(u^{\tau}(x)-u_{\lambda_{0}}^{\tau}(x)\right) \geq 0 .
\end{aligned}
$$
The first inequality is because $\left(\frac{|x|}{\lambda_{0}}\right)^{n-\alpha} u_{\lambda_{0}}(x)>2 u_{\lambda_{0}}(x)>u(x)$ and the monotonicity of $\frac{f(t)}{t^\tau}$. For $x \in \tilde{D}$, it follows from the definition of Kelvin transform that
\begin{equation}\label{eq 28}
 w_{\lambda_0}(x) \geq u_{\lambda_0}(x) \geq \frac{c}{|x|^{n-\alpha}}.
\end{equation}
Let $h(x)=w_{\lambda_0}(x)-\varphi(x)$, where $\varphi(x)$ is given as the auxiliary function in Lemma \ref{aux fun}. For $x \in \{|x|>2R\}\cap D$, we have
\begin{equation}\label{eq 25}
   (-\Delta)^{\alpha/2} h(x)=(-\Delta)^{\alpha/2} \omega_{\lambda_{0}}(x)-(-\Delta)^{\alpha/ 2} \varphi(x) \geq 0 .
\end{equation}
For $x \in B_{\lambda_{0}}^{c}(0) \backslash\left(\{|x|>2R\} \cap D\right)=\left(B_{R}(0) \backslash B_{\lambda_{0}}(0)\right) \cup\left(\overline{B_{2 R}}(0) \backslash B_{R}(0)\right) \cup\left(\{|x|>2R\}\cap \tilde{D}\right)$, we have
$$
\begin{cases}h(x)=w_{\lambda_{0}}(x)-w_{\lambda_{0}}(x)=0, & x \in B_{R} \backslash B_{\lambda_{0}}, \\ h(x)=w_{\lambda_{0}}(x)-\frac{\varepsilon}{|x|^{n-\alpha}}>0, & x \in\left(\overline{B_{2 R}} \backslash B_{R}\right) \cup\left(\{|x|>2R\} \cap \tilde{D}\right).\end{cases}
$$
Thus, 
\begin{equation}\label{eq 26}
    h(x) \geq 0 \quad \text { in }   B_{\lambda_{0}}^{c} \backslash\left(\{|x|>2R\} \cap D\right) .
\end{equation}
We also have
\begin{equation}\label{eq 27}
    \liminf_{|x|\to\infty}h(x)=\liminf_{|x|\to\infty} \left(w_{\lambda_0}(x)-\frac{\varepsilon}{|x|^{n-\alpha}}\right)\geq 0.
\end{equation}
Combining (\ref{eq 25}), (\ref{eq 26}), (\ref{eq 27}) and Theorem \ref{sym mp}, it follows that
\begin{equation}\label{eq 29}
   h(x) \geq 0 \;\;\;\;\text { in } \{|x|>2R\} \cap D .
\end{equation}
It follows from (\ref{eq 28}) and (\ref{eq 29}) that
$$
\omega_{\lambda_{0}}(x) \geq \frac{\varepsilon}{|x|^{n-\alpha}}, \quad \forall |x|>2R .
$$
The proof of Lemma \ref{estimates w} is finished. 
\end{proof}

\section{Direct method of moving spheres}\label{main section} 
In this section, we will apply the maximum principles and a priori estimates developed in Section \ref{pre} to prove Theorem \ref{main thm}. The main ingredient in the proof of Theorem \ref{main thm} is the following result via the method of moving spheres. 
\begin{proposition}[Direct method of moving spheres]\label{moving sphere 12}
Let u be a positive solution of \eqref{main eq} and let $u_{\lambda}$ be the Kelvin transform of $u$, i.e.,
\begin{equation}\label{kelvin}
  u_\lambda(x)=\left(\frac{\lambda}{|x-x_0|}\right)^{n-\alpha}u\left(\frac{\lambda^2}{|x-x_0|^2}(x-x_0)+x_0\right),\quad x\neq x_0.
\end{equation}
If all the assumptions in Theorem \ref{main thm} hold, then precisely one of the following statements holds.  
  \begin{itemize}
      \item[(A)] For every ${x}_0 \in \mathbb{R}^{n}$ and for all $\lambda \in(0,+\infty)$, it holds that $u_{\lambda}(x) \leq  u(x)$ for any $|x-x_0| \geq \lambda$.

      \item[(B)] For every ${x}_0 \in \mathbb{R}^{n}$, there exists $\lambda_{x_0} \in(0,+\infty)$ depending on $x_0$ such that $u_{\lambda_{x_0}}(x) \equiv u(x)$ for any $|x-x_0| \geq \lambda_{x_0}$.
  \end{itemize}
\end{proposition}
\begin{proof}[Proof of Proposition \ref{moving sphere 12}]
The proof contains two steps.

Step 1. We start by considering the center $x_0$ of Kelvin transform to be 0, i.e.,
\begin{equation*}
    u_\lambda(x)=\left(\frac{\lambda}{|x|}\right)^{n-\alpha}u\left(\frac{\lambda^2}{|x|^2}x\right),\quad x\neq 0.
\end{equation*}
We will prove that there exists small $\bar{\lambda} >0$ such that $u_{\lambda}(x) \leq  u(x)$ for all $0<\lambda<\bar{\lambda}$ and $|x| \geq \lambda$. 

Since $u$ positive and $u\in C^{1,1}_{loc}(\mathbb{R}^n)$, there exists $r_{0}>0$ such that
\begin{equation*}
    \nabla_{x}\left(|x|^{\frac{n-\alpha}{2}} u(x)\right) \cdot x>0, \quad \forall 0<|x|<r_{0}.
\end{equation*}
Consequently,
\begin{equation}\label{eq 33}
u_{\lambda}(x)<u(x), \quad \forall 0<\lambda<|x|<r_{0}.
\end{equation}
By Lemma \ref{u at infty}, we know 
\begin{equation*}
    u(x) \geq \frac{1}{c\left(r_{0}\right)|x|^{n-\alpha}}, \quad \forall|x| \geq r_{0}.
\end{equation*}
So there exists small $\bar{\lambda}  \in\left(0, r_{0}\right)$ such that
\begin{equation}\label{eq 34}
 u_{\lambda}(x) \leq  u(x),\quad \forall|x| \geq r_{0}, \quad \forall 0<\lambda< \bar{\lambda}.    
\end{equation}
Combining (\ref{eq 33}) and (\ref{eq 34}), we get Step 1. 

Step 2. 
Define
\begin{equation*}
    \lambda_{0}=\sup \{\mu>0\mid u_{\lambda}(x) \leq  u(x), ~~ \forall 0<\lambda<\mu, ~~  |x| \geq \lambda \}.
\end{equation*}
By Step 1, we have $0< \lambda_{0} \leqslant \infty$. Now we give the following two claims.
\begin{enumerate}
    \item [Claim 1.] If $\lambda_0<\infty$, then $u_{\lambda_0}(x)=u(x)$ for every $|x|\geq \lambda_0$. 
    \item [Claim 2.] If $\lambda_0=+\infty$, then for every center $x_0\in\mathbb{R}^n$ of the Kelvin transform, the corresponding $\lambda_{x_0}=+\infty$.  
\end{enumerate} 
One can see that Claim 2 implies that if $\lambda_0$ is finite, then for every $x_0\in\mathbb{R}^n$ the corresponding $\lambda_{x_0}$ is finite.
If Claim 1 and Claim 2 hold, then the statements (A) and (B) in Proposition \ref{moving sphere 12} can be reduced to the following statement: either $\lambda_{0}$ is infinite or $\lambda_{0}$ is finite. 
The reduced statement is always true. Therefore, in Step 2 it suffices to give proofs of Claim 1 and Claim 2.

First, we prove Claim 1. By the definition of $\lambda_0$, we have
\begin{equation*}
    u_{\lambda_0}(x) \leq  u(x), \quad \forall|x| \geq \lambda_{0}.
\end{equation*}
Lemma \ref{step 2.1} implies that one of the following statements holds: 
either 
\begin{equation}\label{eq 1}
u(x)>u_{\lambda_0}(x), \quad \forall \left|x\right|>\lambda_0,
\end{equation}
or else 
\begin{equation}\label{eq 13}
u(x) \equiv u_{\lambda_0}(x), \quad \forall \left|x\right|>\lambda_0 .
\end{equation}
If (\ref{eq 13}) holds, then Claim 1 is true. Therefore, we only need to rule out (\ref{eq 1}), which completes the proof of the claim. Let $w(x)=u(x)-u_\lambda(x)$ for $x\neq 0$. The idea is to illustrate that if (\ref{eq 1}) holds, then there exists $\delta>0$ such that
\begin{equation}\label{eq 4}
w_{\lambda}(x) \geq 0, \quad \forall \lambda \in\left[\lambda_{0}, \lambda_{0}+\delta\right], \quad \forall|x| \geq \lambda .
\end{equation}
That is, the sphere $\partial B_{\lambda_0}(0)$ can be moved further, which contradicts the supremum definition of $\lambda_{0}$. We will prove (\ref{eq 4}) by contradiction.

Assume that (\ref{eq 4}) does not hold. Let $\delta_{k}=\frac{1}{k}$  for any $k \in N_{+}$. Since $\liminf\limits_{|x|\to\infty}w_\lambda(x)\geq 0$, one can see that there exist $\lambda_{k} \in\left[\lambda_{0}, \lambda_{0}+\delta_{k}\right]$ and $\left|x_{k}\right|>\lambda_{k}$ such that
\begin{equation}\label{eq 31}
w_{\lambda_k}\left(x_{k}\right)=\min _{\substack{\lambda_{0} \leq  \lambda \leq  \lambda_{0}+\delta_{k} \\|x|  \geq \lambda}} w_{\lambda}(x)<0.
\end{equation}
We then show that $\left\{x_{k}\right\}$ is a bounded sequence. For each $\lambda_{k} \in\left[\lambda_{0}, \lambda_{0}+\delta_{k}\right]$ and $|x|>2 R$, it follows from Lemma \ref{estimates w} and definition of Kelvin transform that
\begin{equation*}
\begin{aligned}
w_{\lambda_{k}}(x) & =u(x)-u_{\lambda_{0}}(x)+u_{\lambda_{0}}(x)-u_{\lambda_{k}}(x) 
\geq \frac{\varepsilon}{|x|^{n-\alpha}}+c \frac{\lambda_{0}^{n-\alpha}-\lambda_{k}^{n-\alpha}}{|x|^{n-\alpha}} .
\end{aligned}
\end{equation*}
For sufficiently large $k$, it holds that
\begin{equation*}
0>c\left(\lambda_{0}^{n-\alpha}-\lambda_{k}^{n-\alpha}\right)>-\frac{\varepsilon}{2} .
\end{equation*}
Thus
\begin{equation*}
\omega_{\lambda_{k}}(x) \geq \frac{\varepsilon / 2}{|x|^{n-\alpha}}>0 \;\;\;\;\text { for large } k \text { and }|x|>2 R \text {. }
\end{equation*}
This implies $\left\{x_{k}\right\}$ is a bounded sequence. Suppose its convergent subsequence (still denote as $\{x_{k}\}$) converges to $\bar{x}$. Next we show that 
\begin{equation}\label{eq 30}
|\bar{x}|\neq \lambda_0.    
\end{equation}
If $|\bar{x}|=\lambda_0$, then without loss of generality, we assume $|x_k|<3\lambda_k$ for any $k$.
Denote $d_{k}=\operatorname{dist}\left({x}_{k}, \partial B_{\lambda_{k}}(0)\right)$. 
By Theorem \ref{MP 1} and scaling process, it follows that
\begin{equation}\label{eq 8}
\begin{aligned} (-\Delta)^{\alpha / 2} w_{\lambda_{k}}\left(x_{k}\right) \leq  C_{n,\alpha} \frac{w_{\lambda_{k}}\left(x_{k}\right)}{d_{k}^{\alpha}} \leq  C_{n,\alpha} w_{\lambda_{k}}\left(x_{k}\right) \frac{1}{\left|x_{k}\right|^{\alpha}}. \end{aligned}    
\end{equation}
On the other hand, one can see that
\begin{equation*}
\begin{aligned}
(-\Delta)^{\frac{\alpha}{2}} w_{\lambda_{k}}\left(x_{k}\right)=
\frac{f\left(u\left(x_{k}\right)\right)}{u^{\tau}\left(x_{k}\right)} u^{\tau}\left(x_{k}\right) -\frac{f\left(\big(\frac{|x_{k}|}{\lambda_k}\big)^{n-\alpha} u_{\lambda_{k}}\left(x_{k}\right)\right)}{\left(\big(\frac{|x_{k}|}{\lambda_k}\big)^{n-\alpha} u_{\lambda_{k}}\left(x_{k}\right)\right)^{\tau}} u^{\tau}_{\lambda_{k}}\left(x_{k}\right).
\end{aligned}
\end{equation*}
Since $u\left(x_{k}\right)<u_{\lambda_{k}}\left(x_{k}\right)<\left(\frac{\left|x_{k}\right|}{\lambda_k}\right)^{n-\alpha} u_{\lambda_{k}}\left(x_{k}\right)$, we have
\begin{equation}\label{eq 19}
\begin{aligned}
(-\Delta)^{\frac{\alpha}{2}} w_{\lambda_{k}}\left(x_{k}\right) 
&\geq\frac{f\left(\big(\frac{\left|x_{k}\right|}{\lambda_k}\big)^{n-\alpha} u_{\lambda_{k}}\left(x_{k}\right)\right)}{\left(\big(\frac{\left|x_{k}\right|}{\lambda_k}\big)^{n-\alpha} u_{\lambda_{k}}\left(x_{k}\right)\right)^{\tau}} (u^{\tau}\left(x_{k}\right)-u^{\tau}_{\lambda_{k}}\left(x_{k}\right)) \\
& \geq \tau \frac{f\left(\big(\frac{\left|x_{k}\right|}{\lambda_k}\big)^{n-\alpha} u_{\lambda_{k}}\left(x_{k}\right)\right)}{\left(\big(\frac{\left|x_{k}\right|}{\lambda_k}\big)^{n-\alpha} u_{\lambda_{k}}\left(x_{k}\right)\right)^{\tau}}u_{\lambda_{k}}^{\tau-1}\left(x_{k}\right) w_{\lambda_{k}}\left(x_{k}\right) \\
& =\tau\left(\frac{\lambda_k}{\left|x_{k}\right|}\right)^{2 \alpha} \frac{f\left(u\left(\frac{\lambda_{k}^{2} x_{k}}{\left|x_{k}\right|^{2}}\right)\right)}{u\left(\frac{\lambda_{k}^{2} x_{k}}{\left|x_{k}\right|^{2}}\right)} w_{\lambda_{k}}\left(x_{k}\right) 
\geq\frac{C_{1}}{\left|x_{k}\right|^{2 \alpha}} w_{\lambda_{k}}\left(x_{k}\right)
\end{aligned}
\end{equation}
where $C_1>0$ is dependent of $\lambda_0,\alpha,n,f$ and independent of $k$. The second inequality is due to the mean value theorem, and the last inequality is due to the local boundedness of $f$. 
Combining (\ref{eq 8}) and (\ref{eq 19}), it follows that
\begin{equation*}
\begin{aligned}
\frac{C_{1}}{\left|x_{k}\right|^{2 \alpha}} w_{\lambda_k}\left(x_{k}\right) \leq C_{n,\alpha} \frac{1}{d_{k}^{\alpha}} w_{\lambda_k}\left(x_{k}\right)  \leq C_{n,\alpha} \frac{1}{\left|x_{k}\right|^{\alpha}} w_{\lambda_k}\left(x_{k}\right) .
\end{aligned}
\end{equation*}
Note that $w_{\lambda_k}\left(x_{k}\right) <0$. This implies that
\begin{equation}\label{eq 9}
\begin{aligned}
C_{1}\lambda_0^{2\alpha} \leq  {C_{1}}\left|x_{k}\right|^{2 \alpha} \leq  d_{k}^{\alpha} \leq  C_{n,\alpha}\left|x_{k}\right|^{\alpha} \leq C_{n,\alpha}\left|\lambda_0+1\right|^{\alpha}, 
\end{aligned}
\end{equation}
where $C_{1}$ is a positive constants independent of $k$. If $|\bar{x}|=\lambda_{0}$, then $\lim\limits _{k \rightarrow \infty} d_{k}=0$, which contradicts (\ref{eq 9}). Thus (\ref{eq 30}) holds.

If $\lambda_{0}<|\bar{x}| <\infty$, then $w_{\lambda_{0}}(\bar{x})=\lim\limits_{k\to\infty}w_{\lambda_{k}}\left(x_{k}\right)\leq  0$, which contradicts (\ref{eq 1}). This implies that (\ref{eq 31}) cannot be true. Thus, (\ref{eq 4}) holds, and the sphere $\partial B_{\lambda_0}(0)$ can be moved further. But this contradicts the supremum definition of $\lambda_{0}$. Therefore, (\ref{eq 1}) is ruled out, and the proof of Claim 1 is completed.

Now we prove Claim 2. The proof comes from Li-Zhang \cite[Lemma 2.3]{LyZh}. For the convenience of the readers we outline it here. Since $\lambda_0=+\infty$, it follows that
\begin{equation*}
    |x|^{n-\alpha}u(x)\geq \lambda^{n-\alpha} u\left(\frac{\lambda^2}{|x|^2}x\right),\quad \forall\lambda>0,\, \forall |x|\geq \lambda.
\end{equation*}
Thus
\begin{equation*}
    \lim_{|x|\to\infty}  |x|^{n-\alpha}u(x)=+\infty.
\end{equation*}
On the other hand, if $\lambda_{x_0}<\infty$ for some $x_0$, then by Claim 1 we know that $u(x)\equiv u_{\lambda_{x_0}}(x)$. Therefore,
\begin{equation*}
    \lim_{|x|\to\infty}  |x|^{n-\alpha}u(x)=\lim_{|x|\to\infty}  |x|^{n-\alpha}u_{\lambda_{x_0}}(x) =\lim_{|x|\to\infty} \lambda_{x_0}^{n-\alpha}u\left(\frac{\lambda_{x_0}^2}{|x-x_0|^2}(x-x_0)+x_0\right)<\infty.
\end{equation*}
Contradiction. Thus, the proof of Proposition \ref{moving sphere 12} is finished.  
\end{proof}

Now we proceed to prove Theorem \ref{main thm}. The proof is a combination of Proposition \ref{moving sphere 12} and the following result.  
\begin{proposition}[Li-Zhu \cite{LyZ}]\label{prop}
Let $u \in C^{1,1}_{loc}\left(\mathbb{R}^n\right)$ and let $u_\lambda$ be its Kelvin transform. Then we have the following:
\begin{enumerate}
    \item[(C)] If (A) holds, then $u$ is constant.
    \item[(D)] If (B) holds, then there exist $\beta_1>0$ and $\beta_2>0$ such that
\begin{equation}\label{eq 18}
u(x)=\frac{\beta_1}{\left(\left|x-x_0\right|^2+\beta_2^2\right)^{\frac{n-\alpha}{2}}}, \quad\forall x \in \mathbb{R}^n .
\end{equation}
\end{enumerate}
\end{proposition} 
For the convenience of readers, we outline the proof of Proposition \ref{prop} here.

\begin{proof}
First, we prove (C). For $y \in \mathbb{R}^n$, let
\begin{equation*}
\lambda = |z-y|, \quad x = \frac{|x-y|(z-y)}{\lambda}+y,
\end{equation*}
and
\begin{equation*}
t= \frac{|x-y|}{\lambda} .
\end{equation*}
Then consider the Kelvin transform about the center $y \in \mathbb{R}^n$,
\begin{equation*}
\begin{aligned}
w_\lambda(x) & =u(x)-u_\lambda(x) \\
& =u(t(z-y)+y)- \left(\frac{1}{t}\right)^{n-\alpha} u\left(\frac{z-y}{t}+y\right).
\end{aligned}
\end{equation*}
Denote $h(t)=\left(\frac{1}{t}\right)^{n-\alpha} u\left(\frac{z-y}{t}+y\right)-u(t(z-y)+y)$. From (A) we can see that $h(1)=0$ and $h(t)<0$ for every $t>1$, and thus
\begin{equation*}
0 \geq\left.\frac{d}{d t}\right|_{t=1} h(t)=-(n-\alpha) u(z)-2\langle\nabla u(z), z-y\rangle .
\end{equation*}
Hence, for every $y \in \mathbb{R}^n$ with $|z-y| \neq 0$, we have that
\begin{equation*}
\frac{-(n-\alpha) u(z)-2\langle\nabla u(z), z-y\rangle}{|z-y|} \leq  0 .
\end{equation*}
Let $\nu = \frac{z-y}{|z-y|}$, we have that $\langle\nabla u(z), \nu\rangle \geq 0$ by letting $|y| \rightarrow+\infty$. Since $y \in \mathbb{R}^n$ is arbitrary, $\nabla u(z)=0$. This completes the proof of (C).

Then we prove (D). Suppose that $x_0$ is the center of the Kelvin transform and the corresponding sphere for $x_0$ is $\lambda_{x_0}$, that is,
\begin{equation*}
    \left(\frac{\lambda_{x_0}}{|x-x_0|}\right)^{n-\alpha}u\left(\frac{\lambda_{x_0}^2}{|x-x_0|^2}(x-x_0)+x_0\right)\equiv u(x) \quad \text{for any }|x-x_0|\geq \lambda_{x_0}. 
\end{equation*}
Furthermore, suppose that the corresponding sphere for center $0$ is $\lambda_0$, that is,
\begin{equation*}
    \left(\frac{\lambda_0}{|x|}\right)^{n-\alpha}u\left(\frac{\lambda_0^2}{|x|^2}x\right)\equiv u(x) \quad \text{for any }|x|\geq \lambda_0. 
\end{equation*}
Then
\begin{equation}\label{eq 15}
A:=\lim _{\left|x\right| \rightarrow \infty}\left|x\right|^{n-\alpha} u\left(x\right)=\lambda_{x_0}^{n-\alpha} u(x_0), \quad \forall x_0 \in \mathbb{R}^{n} .
\end{equation}
If $A=0$, by considering the arbitrariness of $x_0$, one can observe that $u \equiv 0$. Next, we consider $A >0$. Without loss of generality we assume that $A=1$. For $x$ large,
\begin{equation}\label{eq 16}
u\left(x\right)=\frac{\lambda_0^{n-\alpha}}{\left|x\right|^{n-\alpha}}\left\{u(0)+\frac{\partial u}{\partial x_i}(0) \cdot \frac{\lambda_0^2 x_i}{\left|x\right|^2}+o\left(\frac{1}{\left|x\right|}\right)\right\},
\end{equation}
and
\begin{equation}\label{eq 17}
u\left(x\right)=\frac{\lambda_{x_0}^{n-\alpha}}{\left|x-x_0\right|^{n-\alpha}}\left\{u(x_0)+\frac{\partial u}{\partial x_i}(x_0) \cdot \frac{\lambda_{x_0}^2\left(x_i-(x_0)_i\right)}{\left|x-x_0\right|^2}+o\left(\frac{1}{\left|x\right|}\right)\right\} .
\end{equation}
Combining (\ref{eq 15}), (\ref{eq 16}), (\ref{eq 17}), and our assumption $A=1$, we have
\begin{equation*}
u^{-\frac{n-\alpha+2}{n-\alpha}}(x_0) \cdot \frac{\partial u}{\partial x_i}(x_0)=u^{-\frac{n-\alpha+2}{n-\alpha}}(0) \cdot \frac{\partial u}{\partial x_i}(0)-(n-\alpha) (x_0)_i .
\end{equation*}
It follows that for some $x_0^{\prime} \in \mathbb{R}^{n}$, $\beta_1>0$ and $\beta_2>0$,
\begin{equation*}
u^{-\frac{2}{n-\alpha}}\left(x\right)=\frac{\left|x-x_0^{\prime}\right|^2+\beta_2^2}{\beta_1} .
\end{equation*}
This completes the proof. 
\end{proof}

Finally, we prove Theorem \ref{main thm}.
\begin{proof}[Proof of Theorem \ref{main thm}]
By Propositions \ref{moving sphere 12} and \ref{prop}, if $u \in \mathcal{L}_\alpha \cap C^{1,1}_{loc}\left(\mathbb{R}^n\right)$ is a positive solution to (\ref{main eq}), then either $(C)$ or $(D)$ holds. If $(C)$ holds, there exists $C_0 \geq 0$ such that $u \equiv C_0$ and thus $f \left(C_0\right)=(-\Delta)^{\alpha / 2} C_0=0$. If $(D)$ is true, then (\ref{eq 18}) is a solution to
\begin{equation*}
(-\Delta)^{\alpha / 2} u=f(u).
\end{equation*}
By directly calculating $(-\Delta)^{\alpha / 2} u$, the function $f$ satisfies that for some constant $C >0$, 
\begin{equation*}
f(t)=C \cdot t^{\frac{n+\alpha}{n-\alpha}},  \quad  \forall t \in\left(0, \max _{x \in \mathbb{R}^n} u(x)\right] .
\end{equation*}
This completes the proof of Theorem \ref{main thm}. 
\end{proof}

\section{Applications}\label{application}
In this section, we use the direct method of moving spheres and the local blow up argument to prove Theorem \ref{half thm}. First, we establish the monotonicity of solutions via the direct moving spheres.

\begin{lemma}\label{half lemma}
    If $u\in \mathcal{L}_\alpha\cap C_{loc}^{1,1}(\mathbb{R}^n_+)\cap C_{loc}(\overline{\mathbb{R}^n_+})$ is a positive solution to (\ref{half eq-In}), then $u$ is monotonically increasing in the $x_n$-direction for $x_n>0$.
\end{lemma}

Before giving the proof of Lemma \ref{half lemma}, we apply Lemma \ref{half lemma} to prove Theorem \ref{half thm}.

\begin{proof}[Proof of Theorem \ref{half thm}] 
Assume that $u$ is a positive solution. We will prove that $u$ is bounded in $\overline{\mathbb{R}^n_+}$. Suppose by contradiction that $u$ is not bounded. By the monotonicity of $u$ in the
$x_{n}$-direction, we may assume that there exists $\left\{\bar{x}_{i}\right\} \subset \mathbb{R}^n_+$ such that $\lim\limits_{i\to\infty} |\bar{x}_{i,n}|=\infty$ and $\lim\limits_{i\to\infty} u(\bar{x}_{i})=\infty$. Consider 
\begin{equation*}
    v_{i}(x)=\left(1- |x-\bar{x}_{i}|\right)^{\frac{\alpha}{p-1}}u(x),\;\;\;\; \left|x-\bar{x}_{i}\right| \leq 1.
\end{equation*}
Since $v_i(x)=0$ on $\partial B_{1}(\bar{x}_i)$, there exists $x_{i} \in B_{1}(\bar{x}_i)$ s.t
\begin{equation*}
    v_i(x_i)=\max_{B_{1}(\bar{x}_i)}v_i(x).
\end{equation*}
Let $\sigma_i=\frac{1}{2}(1- |x_i-\bar{x}_{i}|)>0$. Then
\begin{equation*}
(2 \sigma)^{\frac{\alpha}{p-1}} u\left(x_{i}\right)=v_{i}(x_i)\geq v_{i}(\bar{x}_i)=u(\bar{x}_i) \rightarrow \infty.    
\end{equation*}
It follows that $R_i=\sigma_iu^\frac{p-1}{\alpha}(x_i)\to\infty.$ Since
\begin{equation*}
    v_i(x_i)\geq v_i(x)=(2 \sigma)^{\frac{\alpha}{p-1}} u\left(x\right)\geq \sigma^{\frac{\alpha}{p-1}}u(x),\;\;\;\;x\in B_{\sigma_i}(x_i),
\end{equation*}
we see that 
\begin{equation*}
    u(x) \leq 2^{\frac{\alpha}{p-1}} u(x_i),\;\;\;\; x \in B_{\sigma_i}(x_i).
\end{equation*}
Consider 
\begin{equation*}
    w_{i}(y)=\frac{1}{u(x_i)} u\left(x_i+\frac{y}{u(x_i)^{\frac{\alpha}{p-1}}}\right), \;\;\;\; \forall y \in \mathbb{R}^{n}.
\end{equation*}
Then $w_{i}(0)=1$, $w_i(y)\leq 2^{\frac{\alpha}{p-1}}$ for all $|y|<R_i$, and $w_{i}$ satisfies the equation
\begin{equation*}
    (-\Delta)^{\alpha/2}w_i(y)=w_i^p(y),\quad \forall|y|<R_i.
\end{equation*}
By the local Schauder estimates in \cite{CLWX,JLX}, there exists $\varepsilon_0>0$ such that
\begin{equation*}
    \|w_i\|_{C^{\alpha+\varepsilon}(B_{R_i/4})} \leq C(n,\alpha,p),\quad \text{for all $i$ and $\varepsilon\in(0,\varepsilon_0)$}.
\end{equation*}
After passing to a subsequence, we have
\begin{equation*}
    w_i\to w_\infty \quad \text{in }C_{loc}^{\alpha+\varepsilon}(\mathbb{R}^n).
\end{equation*}
Then by \cite[Theorem 1.1]{DJXY}, we have
\begin{equation*}
    (-\Delta)^{\alpha/2}w_\infty=w_\infty^p+b\quad\text{in }\mathbb{R}^n\text{ for some }b\geq 0.
\end{equation*}
By Theorem \ref{main thm}, either $w_\infty\equiv C_0$ for some constant $C_0\geq 0$ satisfying $C_0^p+b=0$, or there exists $C>0$ such that $t^p+b=Ct^\tau$. Then either $w_\infty\equiv 0$ or 
\begin{equation*}
    w_\infty(x)=\frac{\beta_1}{\left(\left|x-x_0\right|^2+\beta_2^2\right)^{\frac{n-\alpha}{2}}}, \quad \forall x \in \mathbb{R}^n,
\end{equation*}
for some $\beta_1, \beta_2>0$ and $x_0 \in \mathbb{R}^n$. The first case contradicts the fact that $w_{\infty}(0)=1$, and the second case contradicts the monotonicity of $w_{\infty}$ in the $x_{n}$-direction. Thus $u$ must be bounded.

Denote $\left(x_{1}, \cdots, x_{n-1}, x_{n}\right)=\left(x^{\prime}, x_{n}\right)$. Let $\{r_{j}\} \subset \mathbb{R}^+$ such that $r_{j}<r_{j+1} \xrightarrow{j \rightarrow \infty}+\infty$. Let $u_{j}\left(x^{\prime}, x_{n}\right)=u\left(x^{\prime}, x_{n}+r_{j}\right)$. Then $u_{j}$ is increasing in the $x_n$-direction and uniformly bounded. By the local Schauder estimates in \cite{CLWX,JLX}, for some $\varepsilon_0>0$ and any compact $\Omega\subset\mathbb{R}^n$ we have 
\begin{equation*}
    \|u_j\|_{C^{\alpha+\varepsilon}(\Omega)}\leq C(n,\alpha,p),\quad \text{for large $j$ and $\varepsilon\in(0,\varepsilon_0)$}.
\end{equation*}
After passing to a subsequence, there exists a positive function $u_\infty$ such that 
\begin{equation*}
   u_j\to u_\infty\quad\text{in }C_{loc}^{\alpha+\varepsilon}\left(\mathbb{R}^{n}\right).
\end{equation*}
It follows from \cite[Theorem 1.1]{DJXY} that $u_\infty$ satisfies 
\begin{equation*}
    (-\Delta)^{\alpha/2}u_\infty=u_\infty^p+b\quad\text{in }\mathbb{R}^n\text{ for some }b\geq 0.
\end{equation*}
Using Theorem \ref{main thm} again, we deduce that 
\begin{equation*}
    u_\infty(x)=\frac{\beta_1}{\left(\left|x-x_0\right|^2+\beta_2^2\right)^{\frac{n-\alpha}{2}}}, \quad \forall x \in \mathbb{R}^n,
     \end{equation*}
for some $\beta_1, \beta_2>0$ and $x_0 \in \mathbb{R}^n$. By Lemma \ref{half lemma}, we get a contradiction. Hence, the proof of Theorem \ref{half thm} is finished. 
\end{proof}

Finally, we give the proof of Lemma \ref{half lemma}.
\begin{proof}[Proof of Lemma \ref{half lemma}]
Define $x_R=(0,\cdots,0,-R)$. Let
\begin{equation*}
    u_{x_R,\lambda}(y)=\left( \frac{\lambda}{|y-x_R|}\right)^{n-\alpha} u\left( \frac{\lambda^2}{|y-x_R|^2}(y-x_R)+x_R\right)
\end{equation*}
be the Kelvin transform of $u$ with respect to the ball $B_\lambda(x_R)$ with center $x_R$ and radius $\lambda>R$. By direct computations, we have
\begin{equation*}
    \begin{cases}
        (-\Delta)^{\alpha/2}u_{x_R,\lambda}(x)=\left( \frac{\lambda}{|x-x_R|}\right)^{n+\alpha-p(n-\alpha)} u_{x_R,\lambda}^p(x),\;\;\;\; &x\in \Omega,\\
        u_{x_R,\lambda}(x)\equiv 0, \;\;\;\; &x\in \Omega^c.
    \end{cases}
\end{equation*}
We want to show that $u_{x_R,\lambda}(x) \geq u(x)$ for any $x \in B_{\lambda}\left(x_{R}\right) \cap \mathbb{R}^n_+$ and $\lambda>R$.

Step 1. We will show that there exists $\lambda_0(R)\in(R,2R)$ such that $u_{x_R,\lambda}(x) \geq u(x)$ for any $x \in B_{\lambda}\left(x_{R}\right) \cap \mathbb{R}^n_+$ and $\lambda\in(R,\lambda_0(R))$.

Let $w_\lambda=u_{x_R,\lambda}-u$. Denote $\Sigma_\lambda^-=\{ x \in B_{\lambda}\left(x_{R}\right) \cap \mathbb{R}^n_+ \mid w_\lambda<0\}$. Then for any $x\in \Sigma_\lambda^-$, we have
\begin{equation*}
\begin{aligned}
    (-\Delta)^{\alpha/2}w_\lambda(x)&=\left( \frac{\lambda}{|x-x_R|}\right)^{n+\alpha-p(n-\alpha)} u_{x_R,\lambda}^p(x) -u^p(x) \\
    & \geq u_{x_R,\lambda}^p(x) -u^p(x) \\
    &=p\xi^{p-1}(u_{x_R,\lambda}(x) -u(x))\\
    &\geq pu^{p-1}(u_{x_R,\lambda}(x) -u(x)),
\end{aligned}
\end{equation*}
where $u_{x_R,\lambda}\leq\xi\leq u$. Thus, $w_\lambda$ satisfies 
\begin{equation}\label{Euq=063}
    \begin{cases}
          (-\Delta)^{\alpha/2}w_\lambda +(-pu^{p-1})w_\lambda\geq 0 \quad &\text{in }\Sigma_\lambda^-, \\
          w_\lambda\geq 0\quad &\text{in }(B_{\lambda}\left(x_{R}\right) \cap \mathbb{R}^n_+)\backslash \Sigma_\lambda^-,\\
          w_{\lambda}=u_{x_R,\lambda} -u=u_{x_R,\lambda} \geq 0\quad &\text{in } B_{\lambda}\left(x_{R}\right) \cap (\mathbb{R}^n_+)^c,\\
          \left(w_{\lambda}\right)_{x_R, \lambda}=-w_{\lambda}\quad &\text{in } B_{\lambda}\left(x_{R}\right) \setminus \{x_R\}.
    \end{cases}
\end{equation}  

{\it Claim:} There exists $\delta>0$ such that $\Sigma_\lambda^-=\emptyset$ for all $R<\lambda<R+\delta$.  
If not, then there exist $\lambda_k \to R$ ($\lambda_k >R$) and $x_k \in \Sigma_{\lambda_k}^-$ satisfying  
$$
w_{\lambda_k}(x_k) = \min_{x \in \Sigma_{\lambda_k}^-} w_{\lambda_k}(x) = \min_{x \in B_{\lambda_k}\left(x_{R}\right) \setminus \{x_R\}} w_{\lambda_k}(x)  <0. 
$$
It follows from Theorem \ref{MP 1} (2) and the scaling argument that 
\[
(-\Delta)^{\alpha/2} w_{\lambda_k}(x_k) \leq C \frac{ w_{\lambda_k}(x_k)}{(d(x_k))^\alpha}, 
\]
where $d(x_k)=\operatorname{dist}\left({x}_{k}, \partial B_{\lambda_{k}}(x_R)\right) \to 0$ as $k \to \infty$. Hence, by the uniform boundedness of $pu^{p-1}$ in $\Sigma_\lambda^-$, 
\[
 (-\Delta)^{\alpha/2}w_{\lambda_k}(x_k)+(-pu^{p-1}(x_k)) w_{\lambda_k}(x_k) \leq w_{\lambda_k}(x_k) \left( \frac{C}{(d(x_k))^\alpha} + (-pu^{p-1}(x_k))  \right) <0 
\]
for all $k$ sufficiently large. This contradicts \eqref{Euq=063}, and so the Claim is true. Thus, Step 1 is established.

Define 
\[
\bar{\lambda}(R)=\sup \{\mu
\mid \mu>R ~ \text{and } w_{\lambda} \geq 0 \quad\text{in } B_{\lambda}\left(x_{R}\right) \cap \mathbb{R}^n_+, ~~  \forall R<\lambda<\mu\}.
\]
By Step 1, $\bar{\lambda}(R)$ is well-defined and $R < \bar{\lambda}(R) \leq \infty$. 

Step 2. We will show that $\bar{\lambda}(R)=\infty$ for all $R>0$.

Suppose by contradiction that $\bar{\lambda}=\bar{\lambda}(R)<\infty$ for some $R>0$. Then $$w_{\bar{\lambda}}\geq 0\quad \text{in } B_{\bar{\lambda}}\left(x_{R}\right) \cap \mathbb{R}^n_+.$$ We will show that $w_{\bar{\lambda}}>0$ in $B_{\bar{\lambda}}\left(x_{R}\right) \cap \mathbb{R}^n_+.$ If the strict positivity does not hold, then there exists $x^{\prime} \in B_{\bar{\lambda}}\left(x_{R}\right) \cap \mathbb{R}^n_+$ such that $w_{\bar{\lambda}}(x')=0$. By the calculation in Step 1, we have
\begin{equation*}
    \begin{cases}
          (-\Delta)^{\alpha/2}w_{\bar{\lambda}}\geq u_{x_R,\bar{\lambda}}^p-u^p\geq 0 \quad &\text{in } B_{\bar{\lambda}}\left(x_{R}\right) \cap \mathbb{R}^n_+,\\
        w_{\bar{\lambda}}\geq 0\quad &\text{in }B_{\bar{\lambda}}\left(x_{R}\right) \cap (\mathbb{R}^n_+)^c,\\
          \left(w_{\bar{\lambda}}\right)_{x_R, \bar{\lambda}}=-w_{\bar{\lambda}}\quad &\text{in } B_{\bar{\lambda}}\left(x_{R}\right) \setminus \{x_R\}.
    \end{cases}
\end{equation*}
By the maximum principle, we have $w_{\bar{\lambda}}=0$ in $\mathbb{R}^{n}$. However, for $y\in B_{\bar{\lambda}}(x_{R})\cap \partial\mathbb{R}^n_+$, we have
\begin{equation*}
    w_{\bar{\lambda}} (y)=u_{x_R, \bar{\lambda}}(y)-u(y)=u_{x_R, \bar{\lambda}}(y)>0,
\end{equation*}
which is a contradiction. Thus 
\begin{equation}\label{eq 32}
    w_{\bar{\lambda}}>0 \quad\text{in } B_{\bar{\lambda}}\left(x_{R}\right) \cap \mathbb{R}^n_+. 
\end{equation}
Hence, for any small $\delta >0$, we have
$$ 0<m_{0}=\min_{B_{\bar{\lambda}-\delta}(x_{R})\cap \mathbb{R}^n_+} w_{\bar{\lambda}}.$$
By the continuity of $w_\lambda$ with respect to $\lambda$, there exists some $0<\delta_1 \ll \delta$ such that for all $\lambda \in[\bar{\lambda}, \bar{\lambda}+\delta_1]$,  we have 
$$w_{\lambda} \geq \frac{m_{0}}{2}>0 \quad\text{in } B_{\bar{\lambda}-\delta}(x_{R})\cap \mathbb{R}^n_+.$$ 
Applying the similar computations as in Step 1 and Theorem \ref{MP 1} (2), we can obtain that
\begin{equation*}
    w_\lambda>0\quad\text{in } B_{\lambda}(x_{R})\cap \mathbb{R}^n_+,\quad\forall \lambda \in[\bar{\lambda}, \bar{\lambda}+\delta_1].
\end{equation*}
This contradicts the definition of $\bar{\lambda}$. Step 2 is established.

By Step 2, we have
\begin{equation*}
    u_{x_R,R+a}(x)\geq u(x),\quad\forall x\in B_{R+a}(x_R)\cap \mathbb{R}^n_+,\quad \forall R,a>0.
\end{equation*}
It follow that, for every $x\in\mathbb{R}^n_+$ and every $a>x_n$,
\begin{equation*}
    u(x)\leq \lim_{R\to\infty}u_{x_R,R+a}(x)=u(x_1,\cdots,x_{n-1},2a-x_n).
\end{equation*}
This implies that 
\begin{equation*}
    u(x_1,\cdots,x_{n-1},s)\leq u(x_1,\cdots,x_{n-1},t),\quad \forall 0<s<t.
\end{equation*}
Thus $u$ is monotonically increasing in $x_n$-direction for $x_n>0$. Lemma \ref{half lemma} is proved. 
\end{proof}


\end{document}